\documentclass{amsart}

\usepackage{subfigure}
\usepackage{amsmath,amssymb, amsthm,amscd}
\usepackage{enumerate}
\usepackage{fancybox}
\usepackage{graphicx,color}
\usepackage{amsfonts}
\usepackage{multicol}
\usepackage[normalem]{ulem} 
\usepackage{hyperref}
\usepackage{cleveref}
\usepackage{thmtools,thm-restate}
\usepackage[all]{xy}
\usepackage{amscd}
\usepackage{layout}

\numberwithin{equation}{section}

\newcommand{\id}{\operatorname{id}}

\newcommand{\C}{\mathbb{C}}
\newcommand{\R}{\mathbb{R}}
\newcommand{\Z}{\mathbb{Z}}
\newcommand{\Pa}{\partial}

\newcommand{\Int}{\operatorname{Int}}
\newcommand{\Crit}{\operatorname{Crit}}
\newcommand{\Critv}{\operatorname{Critv}}
\newcommand{\Ker}{\operatorname{Ker}}

\renewcommand{\Re}{\operatorname{Re}}
\renewcommand{\Im}{\operatorname{Im}}

\newcommand{\Hom}{\operatorname{Hom}}

\newcommand{\disk}{\mathbb{D}}

\theoremstyle{plain}
\newtheorem{theorem}{Theorem}[section]
\newtheorem{corollary}[theorem]{Corollary}
\newtheorem{lemma}[theorem]{Lemma}
\newtheorem{proposition}[theorem]{Proposition}
\newtheorem{problem}[theorem]{Problem}

\theoremstyle{definition}
\newtheorem{definition}[theorem]{Definition}

\newtheorem{remark}[theorem]{Remark}

\title{Realizing crosscap transpositions as monodromies of singular fibrations}
\author{Kenta Hayano}
\email{k-hayano@math.keio.ac.jp}
\address{Department of Mathematics, Faculty of Science and Technology, Keio University,
Yagami Campus, 3-14-1, Hiyoshi, Kohoku-ku, Yokohama, 223-8522, Japan}

\begin{document}


\begin{abstract}
We introduce a new type of singularity for smooth maps from
$4$-manifolds to surfaces, called an \textit{$M$-singularity}, whose critical
locus is a circle contained in a single fiber. 
We show that the monodromy around an $M$-singularity is a crosscap transposition in the mapping class group of a non-orientable surface.
We also introduce \textit{$M$-fibrations}, namely smooth maps whose singularities consist only of $M$-singularities, and prove that relations among crosscap transpositions give rise to such fibrations on non-orientable $4$-manifolds.
We then study handle decompositions associated with $M$-fibrations and
their orientation double coverings. 
In particular, we describe the
attaching circles and framings of the two $2$-handles arising from the
orientation double cover of an $M$-singularity. Using this description,
we construct a closed non-orientable $4$-manifold which admits an
$M$-fibration but admits no Lefschetz fibration. 
We further discuss
singularity-theoretic properties of the local model of an
$M$-singularity, namely its infinite $\mathcal{A}_e$-codimension and
an explicit stable perturbation.
\end{abstract}

\maketitle

\section{Introduction}\label{sec:intro}

It is well known that the mapping class group of an oriented surface is generated by Dehn twists \cite{Dehn1938generationDehntwist,Lickorish1964generationDehntwist}. 
Moreover, Dehn twists naturally appear as monodromies around Lefschetz singularities; see, e.g., \cite{Kas1980handlebodyLF}. 
This correspondence between singular fibers and mapping class groups plays a fundamental role in the study of $4$-manifolds. 
Indeed, relations among Dehn twists give rise to Lefschetz fibrations and pencils, and hence to symplectic $4$-manifolds \cite{GompfTowardTopcharasympmfd}. 
Using this correspondence, various interesting Lefschetz fibrations/pencils and symplectic $4$-manifolds have been constructed and studied via monodromy factorizations in mapping class groups, including those of \cite{Korkmaz2001noncpxLF,BH2024LFsignature}.

On the other hand, the mapping class group of a non-orientable surface is not generated by
Dehn twists alone. It is known that one also needs additional generators such as $Y$-homeomorphisms
or crosscap transpositions \cite{Lickorish1963Yhomeo,Szepietowski2012level2MCG}. Since Lefschetz
singularities realize Dehn twists as local monodromies, it is natural to ask whether these additional
generators can also be realized geometrically as monodromies of singular fibers.
However, realizing these mapping classes as monodromies of isolated singularities appears to be
difficult from a topological viewpoint. 
Indeed, a neighborhood of an isolated (singular) point in a manifold is locally orientable, and hence nearby regular fibers inherit local orientations. 
In contrast, the additional generators above are supported on one-holed Klein bottles, which are non-orientable. 
This suggests that one should consider singularities with non-isolated critical loci in order to realize such mapping classes as monodromies.

In this paper, we introduce a new type of singularity, called an \textit{$M$-singularity}, whose
critical locus is an embedded circle contained in a single fiber.%
\footnote{The symbol "$M$" comes from the fact that a tubular neighborhood of the critical circle is diffeomorphic to the product of a M\"obius band and $\mathbb{C}$; see \Cref{sec:monodromy M-fibration}.}
We show that the monodromy
around an $M$-singularity is a crosscap transposition (\Cref{thm:monodromy M-sing}). 
As in the case of Lefschetz fibrations, this result enables us to construct fibered non-orientable $4$-manifolds from relations among crosscap transpositions.
More precisely, we introduce a class of smooth maps called \textit{$M$-fibrations}, namely smooth maps whose singularities consist only of $M$-singularities, and show that relations among crosscap transpositions give rise to $M$-fibrations on non-orientable $4$-manifolds (\Cref{thm:construction M-fibration D2} and \Cref{cor:construction M-fibration S2,cor:construction M-fibration S2 section}). 
This provides a new construction method for non-orientable $4$-manifolds via mapping class groups of non-orientable surfaces.
It is worth noting that, although the mapping class group of a non-orientable surface cannot be generated solely by Dehn twists or solely by Y-homeomorphisms, it was shown in \cite{LS2017generationMCGcrosscaptrans} that, for genus at least seven, it is generated by crosscap transpositions alone.
From this viewpoint, realizing crosscap transpositions as monodromies already provides a broad source of non-orientable fibrations arising from mapping class group relations.

The orientation double covering of an $M$-fibration naturally induces a related fibration,
called an $\widetilde{M}$-fibration, whose singularities are easier to analyze from the viewpoint
of handle decompositions. 
A key point is that a single $\widetilde{M}$-singularity yields
two $2$-handles whose attaching circles form a characteristic configuration inside a regular fiber
(\Cref{thm:attcircle framing 2-handle tildeM-fibration}). This description enables us to draw Kirby diagrams of the orientation double
coverings of total spaces of $M$-fibrations.
As an application, we construct an $M$-fibration on a closed non-orientable $4$-manifold which admits
no Lefschetz fibration (\Cref{thm:diffeo type tildeX1}). 
This example shows that allowing crosscap transpositions
as monodromies leads to fibered non-orientable $4$-manifolds which cannot be obtained from relations
among Dehn twists.

The organization of this paper is as follows. In Section~2, we review mapping class groups
of surfaces and their generators, namely Dehn twists and crosscap transpositions. In Section~3,
we introduce $M$-singularities, $M$-fibrations, and their orientation double coverings, namely
$\widetilde{M}$-singularities and $\widetilde{M}$-fibrations. We prove that the monodromy around
an $M$-singularity is a crosscap transposition, and show that relations among crosscap transpositions
give rise to $M$-fibrations. In Section~4, we study handle decompositions associated with
$M$-singularities and $\widetilde{M}$-singularities. In particular, we investigate the attaching
circles and framings of the two $2$-handles arising from a $\widetilde{M}$-singularity. 
Finally, in Section~5, we present examples of $M$-fibrations, describe Kirby diagrams of their orientation double coverings, and prove the existence of $M$-fibrations on non-orientable $4$-manifolds which admit no Lefschetz fibrations.
Appendix~A is devoted to singularity-theoretic properties of
the local model of an $M$-singularity, namely the computation
of its $\mathcal{A}_e$-codimension and the construction of a
stable perturbation.

\subsection*{Conventions and notations}

Throughout the paper, we assume that manifolds are smooth and connected unless otherwise noted. 
We use $(x,y)$ to denote the standard real coordinates of $\C$ (i.e.,~$z=x+\sqrt{-1}y$). 
For $r\geq 0$ and $z\in \C$, let $D_r(z)$ be the closed disk in $\C$ with the center $z$ and the radius $r$, and put $\disk_r =D_r(0)$. 
For a smooth map $f:X\to Y$ between manifolds, we denote the set of critical points and values of $f$ by $\Crit(f)\subset X$ and $\Critv(f)\subset Y$, respectively. 

\section{Mapping class groups of surfaces}\label{sec:MCG}

In this section, we review mapping class groups of orientable and non-orientable surfaces, together with Dehn twists and crosscap transpositions. 
We also recall the orientation lifting homomorphism associated with the orientation double covering of a non-orientable surface.

Let $\Sigma_g^b$ (resp.~$N_g^b$) be an oriented (resp.~non-orientable), compact, connected genus-$g$ surface with $b$ boundary components. 
The symbol $b$ is omitted when $b=0$ (i.e.,~$\Sigma_g=\Sigma_g^0$ and $N_g=N_g^0$).
For $g,b\geq 0$, the surface $N_g^b$ can be obtained from $\Sigma_{0}^{b+g}$ by gluing $g$ copies of M\"obius bands along $g$ boundary components. 
Each such attached M\"obius band is indicated by the symbol $\otimes$ as in \Cref{fig:crosscap transposition}, and we refer to both the symbol $\otimes$ and the corresponding M\"obius band itself as a \textit{crosscap}. 
It is well-known that, up to isotopy, a properly embedded arc in a M\"obius band is either boundary-parallel or essential. 
In \Cref{fig:crosscap transposition}, any curve passing through a crosscap is assumed to intersect the corresponding M\"obius band in an essential arc. 

Let $\mathcal{D}(N_g^b)$ (resp.~$\mathcal{D}(\Sigma_g^b)$) be the group consisting of all (resp.~orientation-preserving) self-diffeomorphisms fixing the boundary pointwise. 
We denote the mapping class group $\pi_0(\mathcal{D}(N_g^b))$ and $\pi_0(\mathcal{D}(\Sigma_g^b))$ by $\mathcal{M}(N_g^b)$ and $\mathcal{M}(\Sigma_g^b)$, respectively.
We define multiplication of mapping class groups to be \textit{opposite} to composition of representatives, that is, $[\xi_1]\cdot [\xi_2] =[\xi_2\circ \xi_1]$ for diffeomorphisms $\xi_1,\xi_2$, in order to make monodromy representations homomorphisms.

Let $\rho:\Sigma_{g-1}^{2b}\to N_g^b$ be the orientation double covering. 
Since the image $\rho_\ast(\pi_1(\Sigma_{g-1}^{2b}))$ is the kernel of $w_1\circ \mathrm{Ab}:\pi_1(N_g^b)\to \Z/2\Z$, where $\mathrm{Ab}:\pi_1(N_g^b)\to H_1(N_g^b;\Z)$ is the abelianization and $w_1\in H^1(N_g^b;\Z/2\Z)\cong \Hom(H_1(N_g^b;\Z),\Z/2\Z)$ is the first Stiefel-Whitney class, each self-diffeomorphism $\phi\in \mathcal{D}(N_g^b)$ preserves $\rho_\ast(\pi_1(\Sigma_{g-1}^{2b}))$, and thus has the unique orientation-preserving lift $\widetilde{\phi}\in \mathcal{D}(\Sigma_{g-1}^{2b})$. 
By the covering homotopy property, we can define a well-defined homomorphism $\mathfrak{o}=\mathfrak{o}_g^b:\mathcal{M}(N_g^b)\to \mathcal{M}(\Sigma_{g-1}^{2b})$ by $\mathfrak{o}([\phi])=[\widetilde{\phi}]$, which we call the \textit{orientation lifting homomorphism}. 

A simple closed curve $c\subset \Int(N_g^b)$ is said to be \textit{two-sided} if the normal bundle of $c$ is trivial and \textit{one-sided} otherwise. 
Let $c\subset \Int(N_g^b)$ be a two-sided simple closed curve and $\theta$ be an orientation of a closed tubular neighborhood $\nu c$ of $c$. 
We take an orientation-preserving diffeomorphism $\kappa:\nu c \to \{z\in \C ~|~ 1/2\leq |z|\leq 2\}$ so that $\kappa(c)$ is the unit circle, and a monotone non-increasing function $\varrho:\R_{\geq 0}\to [0,1]$ so that $\varrho(r) = 1$ for $r<1-\epsilon$ and $\varrho(r)=0$ for $r>1+\epsilon$ for a sufficiently small $\epsilon>0$. 
We define the \textit{right-handed Dehn twist} $t_{c,\theta}\in \mathcal{M}(N_g^b)$ along $c$ with respect to $\theta$ as an isotopy class of the diffeomorphism $T_{c,\theta}:N_g^b\to N_g^b$ defined as follows (cf.~\Cref{fig:Dehn twist}): 
\[
T_{c,\theta}(x) = \begin{cases}
\kappa^{-1}(\exp(2\pi\sqrt{-1}\varrho(|z|))z) & (x\in \nu c, z=\kappa (x))\\
x & (\mbox{otherwise}).
\end{cases}
\]
\begin{figure}[htbp]
\includegraphics{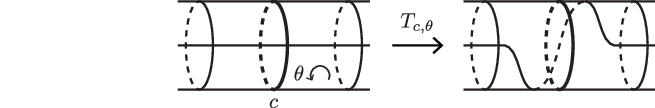}
\caption{Behavior of $T_{c,\theta}$ in $\nu c$. }
\label{fig:Dehn twist}
\end{figure}
Note that $t_{c,\theta}$ does not depend on the choices of $\nu c, \kappa,\varrho,\epsilon$. 
For a simple closed curve $\widetilde{c}\subset \Int(\Sigma_g^b)$, the normal bundle of $\widetilde{c}$ is always trivial and the orientation of $\Sigma_g^b$ induces that of $\nu \widetilde{c}$, in particular one can define the right-handed Dehn twist along $\widetilde{c}$ with respect to this orientation in the same way as above, which we denote by $t_{\widetilde{c}}\in \mathcal{M}(\Sigma_g^b)$. 
Again, let $\rho:\Sigma_{g-1}^{2b}\to N_g^b$ be the orientation double covering. 
The preimage $\rho^{-1}(c)$ is a disjoint union of two circles. 
The restriction of $\rho$ to a neighborhood of one component is orientation-preserving (with respect to $\theta$), while that to a neighborhood of the other is orientation-reversing. 
We denote the former component by $\widetilde{c}_1$ and the latter by $\widetilde{c}_2$. 
It is easy to check that $\mathfrak{o}(t_{c,\theta})$ is equal to $t_{\widetilde{c}_1}t_{\widetilde{c}_2}^{-1}$. 

Let $c \subset \Int(N_g^b)$ be a one-sided simple closed curve, and $d\subset \Int(N_g^b)$ be an oriented two-sided simple closed curve intersecting $c$ at a single point transversely. 
A regular neighborhood $\nu(c\cup d)$ of $c\cup d$ is a one-holed Klein bottle. 
Let $\epsilon>0$ be a sufficiently small number and $K$ be a surface obtained by gluing two copies of M\"obius bands to $\disk_2\setminus \Int (D_\epsilon(1)\sqcup D_\epsilon(-1))$ along $\Pa D_\epsilon(1)\sqcup \Pa D_\epsilon(-1)$. 
We take a diffeomorphism $\kappa':\nu (c\cup d) \to K$ so that $\kappa'(c)$ is isotopic to the core circle of the M\"obius band attached along $\Pa D_\epsilon(1)$ and $\kappa'(d)\cap \left(\disk_2\setminus \Int (D_\epsilon(1)\sqcup D_\epsilon(-1))\right)$ is close to the unit circle with counterclockwise orientation (cf.~the left of \Cref{fig:crosscap transposition}).
We also take a monotone non-increasing smooth function $\varrho':\R_{\geq 0}\to [0,1]$ so that $\varrho'(r)=1$ for $r\leq 4/3$ and $\varrho'(r)=0$ for $r \geq 5/3$. 
We define the \textit{crosscap transposition} $u_{c,d}\in \mathcal{M}(N_g^b)$ as an isotopy class of the diffeomorphism $U_{c,d}:N_g^b\to N_g^b$ satisfying: 
\[
U_{c,d}(x) = \begin{cases}
{\kappa'}^{-1}\left(\exp(\pi \sqrt{-1}\varrho'(|z|))z\right) & (x\in \nu (c\cup d), \kappa'(x)=z\in \disk_2\setminus \Int (D_\epsilon(1)\sqcup D_\epsilon(-1)))\\
x & (x \not\in \nu (c\cup d)).
\end{cases}
\]
\begin{figure}[htbp]
\includegraphics{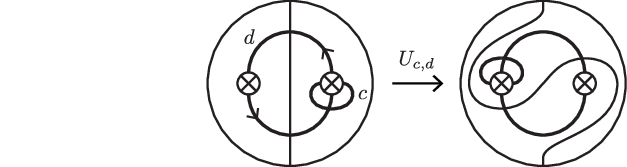}
\caption{Behavior of $U_{c,d}$ in $\nu (c\cup d)\cong K$. }
\label{fig:crosscap transposition}
\end{figure}%
Note that the above formula does not specify the value $U_{c,d}(x)$ for $x$ lying in the two crosscaps. 
Nevertheless, the element $u_{c,d}\in \mathcal{M}(N_g^b)$ is uniquely determined since the mapping class group $\mathcal{M}(N_1^1)$ is trivial (\cite[Theorem~3.4]{Epstein1966curve2mfd}). 
Note also that $u_{c,d}$ does not depend on the choices of $\nu (c\cup d), \kappa',\varrho'$. 
It is easy to check that the square $u_{c,d}^2$ is equal to $t_{\delta,\theta}$, where $\delta = \kappa^{-1} (\Pa \disk_2)\subset \nu (c\cup d)$ and $\theta$ is the pullback by $\kappa$ of the standard orientation of $\disk_2$. 


\section{$M$-fibrations and their monodromies}\label{sec:monodromy M-fibration}

In this section, we introduce $M$-singularities and $M$-fibrations, together with their orientation double coverings, namely $\widetilde{M}$-singularities and $\widetilde{M}$-fibrations. 
We define their monodromies and show that the monodromy around each singular value is a crosscap transposition. 
We also explain how to construct $M$-fibrations from relations in mapping class groups.

Let $\mathfrak{m},\mathfrak{a}:\Z\times \R^2\to \R^2$ be the $\Z$-actions on $\R^2$ defined by $\mathfrak{m}(n,(r,s)) = ((-1)^nr,s+n)$ and $\mathfrak{a}(n,(r,s)) = (r,s+n)$, respectively. 
We denote the orbit space of the action $\mathfrak{m}$ (resp.~$\mathfrak{a}$) by $M$ (resp.~$A$), and let $[r,s]_m\in M$ and $[r,s]_a\in A$ the points represented by $(r,s)\in \R^2$.  
The quotient maps $\pi_m:\R^2\to M$ and $\pi_a:\R^2\to A$ are the universal coverings, in particular these induce the local coordinates, which we denote by $r,s$. 
Note that the map $\pi:A\to M$ defined by $\pi([r,s]_a) = [r,2s]_m$ is the orientation double covering. 
We define a map $F_m:M\times \C\to \C$ by $F_m([r,s]_m,z)=r\exp(\pi\sqrt{-1}s) + z^2$, and put $F_a=F_m\circ (\pi\times \id_{\C}):A\times \C\to \C$ (i.e.,~$F_a([r,s]_a,z)=r\exp(2\pi\sqrt{-1}s)+z^2$). 
It is easy to check that $\Crit(F_m)$ and $\Critv(F_m)$ are respectively equal to $\{([0,s]_m,0)\in M\times \C~|~ s\in \R\}\cong S^1$ and $\{0\}\subset \C$. 

\begin{definition}\label{def:M- tilde{M}-singularity}

Let $f:X\to Z$ be a smooth map from a $4$-manifold $X$ to a surface $Z$. 
A connected component $S\subset \Crit(f)$ or a point in it is called an \textit{$M$-singularity} of $f$ if
\begin{enumerate}

\item 
$S$ is an embedded circle in $X$, 

\item 
$f(S)$ consists of a single point $q$, and 

\item 
there exist a diffeomorphism $\Phi:\nu S \to M\times \C$ defined on a tubular neighborhood $\nu S$ of $S$ and a complex chart $(V,\psi)$ around $q$ satisfying $\psi\circ f\circ \Phi^{-1} = F_m$. 

\end{enumerate}

\noindent
Similarly, $S$ or a point in it is called an \textit{$\widetilde{M}$-singularity} if it satisfies (1), (2), and (3) with $F_m$ replaced by $F_a$. 

\end{definition}

\noindent
Note that $X$ is non-orientable if $f:X\to Z$ has an $M$-singularity, and the composition $f\circ \mathcal{P}$ has an $\widetilde{M}$-singularity, where $\mathcal{P}:\widetilde{X}\to X$ is the orientation double covering.

\begin{definition}\label{def:M-fibration}

Suppose that $X$ and $Z$ are compact. 
A smooth map $f:X\to Z$ is called an \textit{$M$-fibration} if it satisfies the following conditions:

\begin{enumerate}[(I)]

\item 
$\Pa X = f^{-1}(\Pa Z)$, 

\item 
$\Crit(f)$ is contained in $\Int X$ and consists of $M$-singularities, 

\item 
each fiber of $f$ contains at most one component of $\Crit(f)$. 

\end{enumerate}

\noindent
The genus of a regular fiber of $f$ is called the \textit{genus} of $f$. 
We also define an \textit{$\widetilde{M}$-fibration} as a smooth map $f$ satisfying (I), (III), and (II) with $M$-singularities replaced with $\widetilde{M}$-singularities. 
As before, the composition of an $M$-fibration and the orientation double covering of the total space is an $\widetilde{M}$-fibration. 

\end{definition}

Let $f:X\to Z$ be a genus-$g$ $M$-fibration with $\Crit(f)\neq \emptyset$ and $q_0\in Z\setminus \Critv(f)$. 
As in the case of Lefschetz fibrations, we can define a \textit{monodromy representation} $\rho_f:\pi_1(Z\setminus \Critv(f),q_0)\to \mathcal{M}(N_g)$ of $f$ as follows. 
First, fix a diffeomorphism $\phi:N_g \to f^{-1}(q_0)$. 
For a loop $a:[0,1]\to Z\setminus \Critv(f)$ based at $q_0$, choose a trivialization $\Psi:[0,1]\times N_g \to E(a^\ast f)$ so that $\Psi(0,x) = (0,\phi(x))$, where $E(a^\ast f)$ is the total space of the pullback of $f$ (as a surface bundle) by $a$, that is, $E(a^\ast f) = \{(t,x)\in [0,1]\times X~|~ a(t)=f(x)\}$. 
We define a diffeomorphism $\psi_1:N_g\to f^{-1}(q_0)$ by $\psi_1(x)=p_2\circ \Psi(1,x)$, and set $\rho_f([a]) = [\phi^{-1}\circ \psi_1]$. 
Let $q\in \Critv(f)$, $\gamma \in Z$ be a path from $q_0$ to $q$ that meets $\Critv(f)$ only at its endpoint $q$, and $\mu$ be the based loop at $q_0$ obtained by following $\gamma$ to a point near $q$, going once around $q$ along a small circle, and then returning to $q_0$ along $\gamma$. 
We call such a loop or the corresponding element in $\pi_1(Z\setminus \Critv(f),q_0)$ the \textit{meridian loop} of $\gamma$. 

\begin{theorem}\label{thm:monodromy M-sing}

The monodromy $\rho_f(\mu)$ of $f$ along $\mu$ is a crosscap transposition. 

\end{theorem}

\begin{lemma}\label{lem:topology fiber Fm}

For $0\neq w \in \C$, the fiber $F_m^{-1}(w)$ is diffeomorphic to the once-punctured Klein bottle. 

\end{lemma}

\begin{proof}[Proof of \Cref{lem:topology fiber Fm}]

Let $p_1:M\times \C\to M$ and $p_2:M\times \C\to \C$ be the projections, and put $\alpha_w=p_1|_{F_m^{-1}(w)}:F_m^{-1}(w)\to M$ and $\beta_w = p_2|_{F_m^{-1}(w)}:F_m^{-1}(w)\to \C$. 
One can easily deduce from the implicit function theorem that $\beta_w$ is a local diffeomorphism at $([r,s]_m,z)\in F_m^{-1}(w)$ unless $r=0$, or equivalently $z^2=w$. 
Let $\eta$ be a square root of $w$. 
For a sufficiently small $\epsilon >0$, the restriction of $\beta_w$ on the preimage $\beta_w^{-1}\left(\overline{\C \setminus (D_\epsilon(\eta)\sqcup D_\epsilon(-\eta))}\right)$ is a proper bijection, and thus a diffeomorphism. 
Hence, it suffices to show that the preimage $\beta_w^{-1}(D_\epsilon(\pm \eta))$ is diffeomorphic to the M\"obius band. 

We define $\phi:M\to \C$ by $\phi([r,s]_m)=r\exp (\pi \sqrt{-1}s)$ and $q_w:\C\to \C$ by $q_w^\pm(\xi)=w-(\pm \eta + \xi)^2$. 
Since $\alpha_w$ is proper and a local diffeomorphism at any point on $\beta_w^{-1}(D_\epsilon(\pm \eta))$, the restriction $\alpha_w|_{\beta_w^{-1}(D_\epsilon(\pm \eta))}$ is a local diffeomorphism on its image $\phi^{-1}(q_w^\pm(\disk_\epsilon))$. 
The function $q_w^\pm|_{\disk_\epsilon}$ is injective and the following holds for any $\xi\in \disk_\epsilon$ (and a sufficiently small $\epsilon>0$):
\[
\Re\left(1 + \frac{\xi\cdot  (q_w^\pm)''(\xi)}{(q_w^\pm)'(\xi)}\right) = \Re\left(1 + \frac{-2\xi}{-2(\pm \eta +\xi)}\right)\geq 1 - \left|\frac{\xi}{\pm \eta+\xi}\right|>0. 
\]
As a standard consequence of the theory of univalent functions, the above inequality implies that $q_w^\pm(\disk_\epsilon)$ is a convex region (see, e.g.,~\cite{Duren1983univalentfunc}).
In particular, there exists a periodic function $\mathfrak{r}:\R\to \R_{>0}$ with period $2$ such that $\phi^{-1}(q_w(\disk_\epsilon))$ is equal to $\{[r,s]_m\in M~|~ -\mathfrak{r}(s+1)\leq r \leq \mathfrak{r}(s)\}$, which is diffeomorphic to the M\"obius band. 
\end{proof}

\begin{remark}\label{rem:branched cover alpha}
One can also obtain \Cref{lem:topology fiber Fm} by showing that $\alpha_w:F_m^{-1}(w)\to M$ is a double branched covering branched at a single point.
The same argument applies to $F_a$. 
Namely, the restriction of the projection to the first component, $\widetilde{\alpha}_w:F_a^{-1}(w) \to A$, is a double branched covering branched at two points. 
This observation will play an important role later in describing a Kirby diagram of the orientation double covering of the total space of an $M$-fibration.
\end{remark}

\begin{proof}[Proof of \Cref{thm:monodromy M-sing}]
Let $\Phi:\nu S\to M\times \C$ and $(V,\psi)$ be a diffeomorphism and a complex chart around $q$ satisfying $\psi\circ f \circ \Phi^{-1}=F_m$.
Since $X$ is compact and $f^{-1}(q)$ contains one component of $\Crit(f)$, we may assume that $V\cap \Critv(f)=\{q\}$.  
We define a path $a:[0,1]\to Z$ by $a(t) = \psi(\exp (2\pi \sqrt{-1}t))$. 
It is enough to show that the monodromy of $f$ along $a$ is a crosscap transposition. 
In what follows, we identify $V$ with $\C$ via $\psi$. 

As in the proof of \Cref{lem:topology fiber Fm}, one can show that the restriction of the projection to the second component, $\beta_w:F_m^{-1}(w)\cap (M\times (\C\setminus \Int(\disk_2))) \to \C\setminus \Int(\disk_2)$, is a diffeomorphism for $w\in \disk_1$ (even when $w=0$).
In particular, we can define a trivialization $\Psi:\disk_1\times (\C\setminus\Int (\disk_2))\to F_m^{-1}(\disk_1)\cap (M\times (\C\setminus \Int(\disk_2)))$ of $F_m$ (as a surface bundle) by $\Psi(w,z)=\beta_w^{-1}(z)$. 
We take a Riemannian metric $g$ of $f^{-1}(\disk_1)$ so that $\Phi_\ast g$ coincides with the pushforward of the product metric of $\disk_1\times (\C\setminus \Int(\disk_2))$ by $\Psi$ on some neighborhood of $\Psi (\disk_1\times \Pa \disk_2)$. 
Since $\Crit(f)\cap f^{-1}(\disk_1)=S$, we can define a trivialization $\Theta:\disk_1 \times (f^{-1}(0)\setminus \Phi^{-1}(M\times \disk_2)) \to f^{-1}(\disk_1)\setminus \Phi^{-1}(M\times \disk_2)$ by $\Theta(w,x) = c_{w,x}(1)$, where $c_{w,x}:[0,1]\to X$ is a lift with $c_{w,x}(0)=x$ of the segment $t \mapsto t w \in \disk_1$ by the horizontal distribution $(\Ker df)^\perp$.
Note that, by the choice of the metric $g$, $\Theta$ is equal to $\Phi^{-1}\circ\Psi\circ (\id_{\disk_1}\times (\beta_0\circ \Phi))$ on a neighborhood of $\disk_1\times \Phi(F_m^{-1}(0)\cap (M\times \Pa\disk_2))$.

Let $\varrho':\R_{\geq 0} \to [0,1]$ be a function taken when defining a crosscap transposition in \Cref{sec:MCG}. 
We define a smooth function $w:[0,1]\times \C \to \C$ by
\[
w(t,z) = \exp(2\pi\sqrt{-1}t) - \exp(2\pi\sqrt{-1}\varrho'(|z|)t) z^2. 
\]
Note that $w(t,z)=0$ if and only if $z=\pm 1$. 
We further define $\widetilde{s}:[0,1]\times \R\times (\C\setminus \{\pm 1\}) \to \R$ as follows: 
\begin{align*}
\widetilde{s}(t,s,z) =& s+\int_0^t\frac{1}{\pi}\Im \left(\frac{\frac{\Pa w}{\Pa t}(t,z)}{w(t,z)}\right)dt\\
=&s + \int_0^t\frac{1}{\pi}\Im \left(\frac{2\pi\sqrt{-1}\exp(2\pi\sqrt{-1}t) - 2\pi \sqrt{-1}\varrho'(|z|)\exp(2\pi \sqrt{-1}\varrho'(|z|)t)z^2}{\exp(2\pi\sqrt{-1}t) - \exp(2\pi \sqrt{-1}\varrho'(|z|)t)z^2}\right)dt.
\end{align*}
Since $\widetilde{s}(t,s,z)=s+2t$ when $|z|\leq 4/3$, we can extend $\widetilde{s}$ to the function on $[0,1]\times \R\times \C$, which we denote by the same symbol $\widetilde{s}$. 
For $t\in [0,1]$, $z\in \C$, $r\in \R$ with $|r|=|1-z^2|$, we define $\widetilde{r}(t,r,z)\in \R$ as follows: 
\[
\widetilde{r}(t,r,z) = \begin{cases}
r & (|z| < 4/3)\\
|w(t,z)| & (z\neq \pm 1\land r>0)\\
-|w(t,z)| & (z\neq \pm 1\land r<0).
\end{cases}
\]
The value $\widetilde{r}(r,t,z)$ is indeed well-defined since $|w(t,z)|=|1-z^2|=|r|$ when $|z|\leq 4/3$. 
Lastly, we define $\Lambda:[0,1]\times f^{-1}(1)\to E(a^\ast f)$ as follows: \[
\Lambda(t,x) = \begin{cases}
\left(t,\Theta\left(a(t), p_2(\Theta^{-1}(x))\right)\right) & (x \not\in \Phi^{-1}(M\times \disk_2))\\
\left(t,\Phi\left([\widetilde{r}(t,r,z),\widetilde{s}(t,s,z)]_m,\exp(\pi \sqrt{-1}\varrho'(|z|)t)z\right)\right) & (x = \Phi^{-1}([r,s]_m,z)\mbox{ for }z\in \disk_2).
\end{cases}
\]
One can check that $\Lambda$ is a trivialization of the bundle $a^\ast f$. 
Hence, the monodromy of $f$ along $a$ is represented by $\Lambda_1:x\mapsto \Lambda(1,x)$, whose support is contained in $\beta_1^{-1}(\disk_2)$. 
Since $\beta_1^{-1}(z)$ is sent to $\beta_1^{-1}(\exp(\pi \sqrt{-1}\varrho'(|z|))z)$ for $z\in \disk_2 \setminus (D_\epsilon(1)\sqcup D_\epsilon(-1))$ by this representative, the monodromy of $f$ along $a$ is a crosscap transposition interchanging the two M\"obius bands appearing in the proof of \Cref{lem:topology fiber Fm}.
\end{proof}

\subsection*{Construction of $M$-fibrations from relations in mapping class groups}

As in the case of Lefschetz fibrations (see, e.g.,~\cite{Kas1980handlebodyLF,GS19994-mfd}), one can define several types of \textit{monodromy factorizations} as follows.
From a genus-$g$ $M$-fibration $f:X\to D^2$, one can obtain a sequence $\xi_1,\ldots, \xi_n$ of crosscap transpositions in $\mathcal{M}(N_g)$, each of which is a monodromy of $f$ along a meridian loop of a path in a Hurwitz system.
In particular, its product is a monodromy of $f$ along a loop homotopic to $\Pa D^2$. 
The factorization $\xi_1\cdots \xi_n$ of a monodromy along $\Pa D^2$ into a product of crosscap transpositions is called a \textit{monodromy factorization} of $f$. 
In the same manner, we can also obtain a sequence of crosscap transpositions from a genus-$g$ $M$-fibration $f:X\to S^2$, and the product of them is equal to $1\in \mathcal{M}(N_g)$. 
We call the corresponding factorization of $1\in \mathcal{M}(N_g)$ a \textit{monodromy factorization} of $f:X\to S^2$. 
Suppose that a genus-$g$ $M$-fibration $f:X\to S^2$ has sections $\sigma_1,\ldots,\sigma_b:S^2\to X$. 
The monodromy factorization $\xi_1\cdots \xi_n=1$ of $f$ can be lifted (via the capping homomorphism) to a factorization $\widetilde{\xi}_1\cdots \widetilde{\xi}_n=t_{\delta_1,\theta_1}^{a_1}\cdots t_{\delta_b,\theta_b}^{a_b}$ in $\mathcal{M}(N_g^b)$, where $\widetilde{\xi}_i$ is a crosscap transposition in $N_g^b$, $\delta_1,\ldots,\delta_b\subset N_g^b$ are the (isotopy classes of) simple closed curves parallel to the boundary components of $N_g^b$, $\theta_i$ is some orientation of $\nu\delta_i$, and $a_i$ is an integer which coincides with the Euler number of the normal bundle of $\sigma_i(S^2)$ (with some orientation) up to sign. 
We call this factorization (of $t_{\delta_1,\theta_1}^{a_1}\cdots t_{\delta_b,\theta_b}^{a_b}$) a \textit{monodromy factorization} of $f$ and $\sigma_1,\ldots,\sigma_b$. 

\begin{theorem}\label{thm:construction M-fibration D2}

For any sequence $\xi_1,\ldots, \xi_n\in \mathcal{M}(N_g)$ of crosscap transpositions, there exists an $M$-fibration $f:X\to D^2$ whose monodromy factorization is $\xi_1\cdots \xi_n$. 

\end{theorem}

\begin{proof}
It suffices to show the theorem for the case $n=1$. 
Indeed, if we can construct an $M$-fibration $f_i:X_i\to D^2$ for each $i=1,\ldots, n$ such that $f_i$ has a single critical value and the monodromy around it is $\xi_i$ (under some identification of a reference fiber of $f_i$ with $N_g$), we can obtain a desired $M$-fibration by taking a boundary fiber sum of $f_1,\ldots, f_n$. 

In the proof of \Cref{thm:monodromy M-sing}, we constructed a diffeomorphism $\Lambda_1:f^{-1}(1)\to f^{-1}(1)$, and regarded $F_m^{-1}(1)\cap (M\times \disk_2)$ as a subset of $f^{-1}(1)$. 
The restriction $\Lambda_1|_{F_m^{-1}(1)\cap (M\times \disk_2)}$ does not depend on $f$. 
By abuse of notation, we denote this restriction by the same symbol $\Lambda_1$. 

Since $\xi_1$ is a crosscap transposition, we can take an embedding $\iota:F_m^{-1}(1)\cap (M\times \disk_2)\to N_g$ so that 
$\xi_1$ is represented by the diffeomorphism such that its support is $\Im \iota$ and it sends $x\in \Im \iota$ to $\iota\circ \Lambda_1\circ \iota^{-1}(x)$. 
Let $\Psi:\disk_1\times (\C\setminus \Int(\disk_2))\to F_m^{-1}(\disk_1)\cap \left(M\times (\C\setminus \Int(\disk_2))\right)$ be the trivialization defined in the proof of \Cref{thm:monodromy M-sing}. 
We take a diffeomorphism $\eta:\Pa \disk_2\to \Pa \overline{N_g\setminus \Im \iota}$. 
Let $X$ be a manifold obtained by gluing $F_m^{-1}(\disk_1)\cap (M\times \disk_2)$ to $\disk_1\times \overline{N_g\setminus \Im \iota}$ by the following diffeomorphism:
\[
(\id_{\disk_1}\times \eta)\circ \Psi^{-1}:F_m^{-1}(\disk_1)\cap (M\times \Pa \disk_2) \to \disk_1\times \Pa \overline{N_g\setminus \Im \iota}. 
\]
Since $\Psi$ is a trivialization of $F_m$, we can define $f:X\to \disk_1$ so that $f$ is equal to $F_m$ on $F_m^{-1}(\disk_1)$ and the projection to the first component on $\disk_1\times \overline{N_g\setminus \Im \iota}$. 
The map $f$ is an $M$-fibration which has the unique critical value $0\in \disk_1$. 
In the same way as in the proof of \Cref{thm:monodromy M-sing}, we can show that the monodromy of $f$ along $\Pa \disk_1$ is $\xi_1$.
\end{proof}

\begin{corollary}\label{cor:construction M-fibration S2}

Let $\xi_i\in \mathcal{M}(N_g)$ ($i=1,\ldots,n$) be a crosscap transposition such that $\xi_1\cdots \xi_n=1$. 
There exists an $M$-fibration $f:X\to S^2$ whose monodromy factorization is $\xi_1\cdots \xi_n=1$. 

\end{corollary}

\begin{proof}
By \Cref{thm:construction M-fibration D2}, there exists an $M$-fibration $f':X'\to D^2$ whose monodromy factorization is $\xi_1\cdots \xi_n$. 
By the assumption, a monodromy of $f'$ along $\Pa D^2$ is trivial, and thus $f'|_{\Pa D^2}$ is a trivial $N_g$-bundle. 
We can obtain a desired $M$-fibration $f:X\to S^2$ by gluing $D^2\times N_g$ to $X'$ by an isomorphism $\Pa D^2\times N_g\to \Pa X'$ as an $N_g$-bundle.
\end{proof}

\begin{corollary}\label{cor:construction M-fibration S2 section}

Let $\widetilde{\xi}_1,\ldots, \widetilde{\xi}_n\in \mathcal{M}(N_g^b)$ be crosscap transpositions satisfying $\widetilde{\xi}_1\cdots \widetilde{\xi}_n = t_{\delta_1,\theta_1}^{a_1}\cdots t_{\delta_b,\theta_b}^{a_b}$ for some $a_1,\ldots, a_b\in \Z$. 
There exist a genus-$g$ $M$-fibration $f:X\to S^2$ and sections $\sigma_1,\ldots, \sigma_b:S^2\to X$ of it such that the corresponding monodromy factorization is $\widetilde{\xi}_1\cdots \widetilde{\xi}_n = t_{\delta_1,\theta_1}^{a_1}\cdots t_{\delta_b,\theta_b}^{a_b}$. 

\end{corollary}

\begin{proof}
In the same way as for \Cref{thm:construction M-fibration D2,cor:construction M-fibration S2}, we can construct a smooth map $f:X\to S^2$ such that $\Crit(f)$ consists of $M$-singularities and the restriction $f|_{f^{-1}(S^2\setminus \Critv(f))}$ is an $N_g^b$-bundle.
It is easy to see that the boundary component of $X$ corresponding to that of $N_g^b$ parallel to $\delta_i$ is an $S^1$-bundle over $S^2$ such that the absolute value of the Euler number is $|a_i|$. 
We obtain a desired $M$-fibration by capping each boundary component by a $D^2$-bundle over $S^2$. 
\end{proof}

\begin{remark}\label{rem:uniqueness M-fibration}

In view of the corresponding uniqueness result for Lefschetz fibrations, it is natural to expect that the isomorphism class of an $M$-fibration is determined by its monodromy factorization if its genus is at least three (cf.~\cite{Gramain1973homotopytypediffeogrp}). 
The author indeed expects that such a uniqueness statement holds. 
However, the main purpose of this paper is to introduce $M$-fibrations and to clarify the difference between $M$-fibrations and Lefschetz fibrations, as illustrated by \Cref{thm:diffeo type tildeX1}.
For this purpose, the uniqueness problem is not needed, and we leave it for future work.

Let us only point out one issue which seems to arise in proving such a uniqueness statement. 
In the Lefschetz case, one uses the fact that a Dehn twist determines the isotopy class of its underlying simple closed curve (cf.~\cite{Matsumoto1996LFgenus2}). 
Thus, in order to adapt the same argument to $M$-fibrations, one would first need to understand whether a crosscap transposition determines, in an appropriate sense, the isotopy class of the pair of curves used to define it.

\end{remark}

\section{Handle decompositions associated with $M$-fibrations and their orientation double coverings}

In this section, we study handle decompositions associated with $M$-fibrations and their orientation double coverings. 
In particular, we describe the attaching circles and framings of the $2$-handles arising from $\widetilde{M}$-singularities, which enables us to draw Kirby diagrams of the total spaces.

Let $f:X\to \disk_1$ be an $M$-fibration. 
We define a function $h:\disk_1\to \R$ by $h(z)=|z|$. 
Put $\Crit(f)=S_1\sqcup \cdots \sqcup S_n$, where $S_k$ is a connected component (i.e.,~an $M$-singularity).  
Without loss of generality, we can assume that $f(S_k) = \{\frac{1}{2}\exp(2\pi\sqrt{-1}k/n)\} =:\{q_k\}$, and there exist a diffeomorphism $\Phi_k:\nu S_k \to M\times \C$ and a complex chart $\psi_k:V_k\to \C$ around $q_k$ such that $\psi_k\circ f\circ \Phi_k^{-1} = F_m$ and $h\circ \psi_k^{-1}(z)=-\Re (z)+1/2$. 

\begin{proposition}\label{prop:Morse function M-fibration}

The restriction of $h\circ f$ to $X\setminus f^{-1}(0)$ is a Morse function such that the critical point set $\Crit(h\circ f)$ is $\{\Phi_k^{-1}\left([0,1/2]_m,0\right)\in X ~|~ k=1,\ldots, n\}$ and the index of each critical point is $2$. 

\end{proposition}

\begin{proof}
Since $f$ is a submersion on $X\setminus (S_1\sqcup \cdots \sqcup S_n)$, $\Crit(h\circ f)$ is contained in $S_1\sqcup \cdots \sqcup S_n$. 
The following holds under the local coordinates $(r,s,x,y)$ around a point in $\nu S_k$:
\begin{equation}\label{eqn:local description hf}
h\circ f\circ \Phi_k^{-1}(r,s,x,y) = -r \cos(\pi s)-x^2+y^2+1/2. 
\end{equation}
The gradient of the above function is $(-\cos(\pi s),\pi r\sin(\pi s),-2x,2y)$, and therefore $\nu S_k$ has a unique critical point $\Phi_k^{-1}([0,1/2]_m,0)$ of $h\circ f$. 
The Hessian matrix of the above function at this critical point is $\left(\begin{array}{c|c}
\begin{smallmatrix}
0 & \pi\\
\pi&0
\end{smallmatrix}& 0\\ \hline
0&\begin{smallmatrix}
-2&0\\
0&2
\end{smallmatrix}
\end{array}\right)$, which is non-degenerate and has two negative eigenvalues.
\end{proof}

\begin{corollary}\label{cor:handle decomposition M-fibration}

Let $\epsilon>0$ be sufficiently small real number. 
The total space $X$ of a genus-$g$ $M$-fibration $f:X\to \disk_1$ is obtained by gluing $n$ $2$-handles to $f^{-1}(\disk_{1/2-\epsilon})$ along the following attaching circles ($k=1,\ldots, n$): 
\[
\Phi_k^{-1}\left(\left\{([\cos(\pi s),s]_m,x)\in M\times \C~|~s,x\in \R, x^2+\cos^2(\pi s)=\epsilon\right\}\right)\subset f^{-1}(\disk_{1/2-\epsilon}). 
\]
In particular, $X$ admits a handle decomposition with one $0$-handle, $g$ $1$-handles, and $1+n$ $2$-handles. 

\end{corollary}

Let $\mathcal{P}:\widetilde{X}\to X$ be the orientation double covering. 
Then $\widetilde{f}=f\circ \mathcal{P}$ is an $\widetilde{M}$-fibration such that $\Crit(\widetilde{f})$ is $\widetilde{S}_1\sqcup \cdots \sqcup \widetilde{S}_n$, where $\widetilde{S}_k = \mathcal{P}^{-1}(S_k)$.
Furthermore, by \Cref{prop:Morse function M-fibration}, we obtain:

\begin{corollary}\label{cor:handle decomposition tildeM-fibration}

The manifold $\widetilde{X}$ is obtained by gluing $2n$ $2$-handles to $\widetilde{f}^{-1}(\disk_{1/2-\epsilon})$ along the following attaching circles ($k=1,\ldots, n$): 
\begin{align*}
C_1=&\widetilde{\Phi}_k^{-1}\left(\left\{([\cos(2\pi s),s]_a,x)\in A\times \C~|~x\in \R, s\in (0,1/2), x^2+\cos^2(2\pi s)=\epsilon\right\}\right),\mbox{ and}\\
C_2=&\widetilde{\Phi}_k^{-1}\left(\left\{([\cos(2\pi s),s]_a,x)\in A\times \C~|~x\in \R, s\in (1/2,1), x^2+\cos^2(2\pi s)=\epsilon\right\}\right),
\end{align*}
where $\widetilde{\Phi}_k:\nu \widetilde{S}_k\to A\times \C$ is a lift of $\Phi\circ \mathcal{P}$ by $(\pi\times \id_{\C}):A\times \C\to M\times \C$ defined at the beginning of \Cref{sec:monodromy M-fibration}.

\end{corollary}

\noindent
In what follows, we will describe configurations and framings of the $2$-handles in $\nu \widetilde{S}_k$ (in order to draw a Kirby diagram of $\widetilde{X}$). 
Applying a suitable scaling transformation, it is enough to consider the $2$-handles associated with the Morse function $\omega := (-\Re)\circ F_a:A\times \C\to \R$ attached to $\omega^{-1}(-1/2)$. 

Let $\widetilde{\alpha}_w:F_a^{-1}(w)\to A$ be the restriction of the projection to the first component. 
When the value $w$ is clear from the context, we write $\widetilde{\alpha}$ instead of $\widetilde{\alpha}_w$ for simplicity. 
As observed in \Cref{rem:branched cover alpha}, for $w\neq 0$, $\widetilde{\alpha}_w$ is a double branched covering branched at the two point $[r_w,s_w]_a,[-r_w,s_w+1/2]_a\in A$, where $r_w>0$ and $s_w\in \R$ satisfy $r_w \exp(2\pi\sqrt{-1}s_w)=w$ (or equivalently, $([r_w,s_w]_a,0)\in F_a^{-1}(w)$).

We first construct a vector field on $\omega^{-1}(-1/2)$ which moves the attaching circles $C_1$ and $C_2$ into $F_a^{-1}(1/2)$.
We put $\eta := \Im \circ F_a$. 
On the open subset $U_1 :=\{([r,s]_a,z)\in \omega^{-1}(-1/2)~|~z\neq 0\}\subset \omega^{-1}(-1/2)$, define a vector field $V_z$ by
\[
V_z:=\frac{y}{2(x^2+y^2)}\partial_x+\frac{x}{2(x^2+y^2)}\partial_y.
\]
By the direct calculation, we can check that $V_z(\omega)=0$ and $V_z(\eta)=1$. 
We also define a vector field $V_r$ on a small neighborhood $U_2$ of $\omega^{-1}(-1/2)\cap (A\times \{0\})$ as follows:
\[
V_r
:=
\sin(2\pi s)\,\partial_r+\frac{\cos(2\pi s)}{2\pi r}\,\partial_s.
\]
This vector field also satisfies $V_r(\omega)=0$ and $V_r(\eta)=1$.
Let $\{\rho_z,\rho_r\}$ be a partition of unity on $\omega^{-1}(-1/2)$ subordinate to the open cover $\{U_1,U_2\}$. Then the vector field $V:=\rho_zV_z+\rho_rV_r$ on $\omega^{-1}(-1/2)$ satisfies $V(\omega)=0$ and $V(\eta)=1$. 

We next examine flows of $V_z,V_r$ and $V$. 
First observe that $V_z$ has no $\partial_r$- or $\partial_s$-component. Hence, for any point $p\in \omega^{-1}(-1/2)$, the projection $\widetilde{\alpha}(p)$ is unchanged along the flow of $V_z$. 
In particular, if $\widetilde{\alpha}(p)$ lies outside $\widetilde{\alpha}(U_2)$, then the flow line of $V$ starting at $p$ coincides with that of $V_z$, and thus $\widetilde{\alpha}$ is constant along such a flow line.

By construction, $\widetilde{\alpha}(U_2)$ is contained in a neighborhood of
\[
B_1\sqcup B_2:=\{[r_{1/2+\sqrt{-1}u},s_{1/2+\sqrt{-1}u}]_a\in A \mid u\in \R\}\sqcup \{[-r_{1/2+\sqrt{-1}u},s_{1/2+\sqrt{-1}u}+1/2]_a\in A \mid u\in \R\}.
\]
On this neighborhood, the map
\[
\Gamma:A\to \C,\qquad [r,s]_a\mapsto r\exp(2\pi \sqrt{-1}s)
\]
is a diffeomorphism onto its image.
Since $V_r$ depends only on $[r,s]_a$, it naturally induces a vector field on $A$. 
This vector field corresponds to $\partial_y$ on $\C$ via $\Gamma$. 
Therefore, for a point $p\in U_2$, the image $\Gamma\circ \widetilde{\alpha}(\phi_t^{V_r}(p))$ moves only in the imaginary direction, i.e., its real part remains constant while its imaginary part increases linearly in $t$, where $\phi_t^{V_r}:A\to A$ is the flow of $V_r$.
Finally, since $V=\rho_z V_z+\rho_r V_r$, it follows that for any point $p\in U_2$, the image $\Gamma\circ \widetilde{\alpha}(\phi_t^{V}(p))$ also moves only in the imaginary direction. Indeed, the contribution of $V_z$ does not change $\widetilde{\alpha}$, while the contribution of $V_r$ changes only the imaginary part under $\Gamma$.

We now define an isotopy $H_t^i:C_i\to \omega^{-1}(-1/2)$ of $C_i$ by $H_t^i(p):=\phi_{-t\,\eta(p)}^V(p)$.
Since $V(\eta)=1$, $\eta(H_t^i(p))$ is equal to $(1-t)\eta(p)$ for all $p\in C_i$ and $t\in [0,1]$, and hence $H_1^i(C_i)\subset F_a^{-1}(1/2)$.

By \Cref{cor:handle decomposition tildeM-fibration}, each of the two attaching circles $C_1,C_2$ does not intersect $F_a^{-1}(1/2+\sqrt{-1}u)$ if $|u|>1/2$, and intersects $F_a^{-1}(1/2+\sqrt{-1}u)$ at two points (resp.~a single point) if $|u|<1/2$ (resp.~$|u|=1/2$). 
In the case $|u|<1/2$, the two intersection points are sent to the same point by $\widetilde{\alpha}$. 
As $u$ varies, the images of these intersection points under $\widetilde{\alpha}$ trace out curves, which are shown as thick solid curves in \Cref{fig:graph branch attcircle}. 
On the other hand, the curves $B_1$ and $B_2$ are depicted in the same figure as dashed curves. 
Note that both the solid and dashed curves are oriented in the direction of increasing $u$. 
As discussed above, for points in $U_2$, the flow of $V$ changes only the imaginary part under the map $\Gamma\circ \widetilde{\alpha}$. 
Thus, the images of such points under $\widetilde{\alpha}$ are dragged along this dashed curve by the flow of $V$.
As a consequence, the images of the curves $H_1^1(C_1), H_1^2(C_2)\subset F_a^{-1}(1/2)$ under $\widetilde{\alpha}:F_a^{-1}(1/2)\to A$ are as shown in \Cref{fig:attcircle annulus1}. 
Moving the branch point on the left-hand side of \Cref{fig:attcircle annulus1} to the right along the thin arrow yields \Cref{fig:attcircle annulus2}. 
A genus-one surface given as a double branched covering at the two points in the same figure is illustrated in \Cref{fig:attcircle genus1surface}. 
\begin{figure}[htbp]
\subfigure[Dashed curves are $B_1\sqcup B_2$, while solid curves are the loci of images of the intersections between the attaching circles and $F_a^{-1}(1/2+\sqrt{-1}u)$ under $\widetilde{\alpha}$ as $u$ varies. ]{\includegraphics{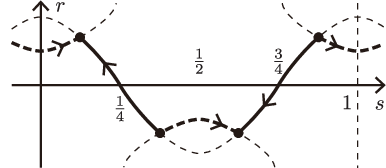}\label{fig:graph branch attcircle}}\hspace{10pt}%
\subfigure[The images of the isotoped attaching circles $H_1^1(C_1)$ and $H_1^2(C_2)$ under $\widetilde{\alpha}$.]{\includegraphics{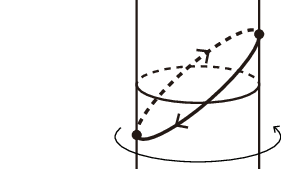}\label{fig:attcircle annulus1}}

\subfigure[This figure is obtained by moving a branch point in \subref{fig:attcircle annulus1} along the thin arrow.]{\includegraphics{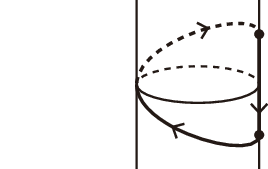}\label{fig:attcircle annulus2}}\hspace{10pt}%
\subfigure[A genus-one surface given as a double branched covering at the two points in \subref{fig:attcircle annulus2}.]{\includegraphics{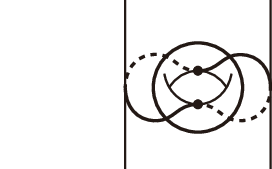}\label{fig:attcircle genus1surface}}
\caption{Configurations of the two attaching circles and the two branch points.}
\label{fig:configuration branch attcircle}
\end{figure}

The curves $H_1^1(C_1)$ and $H_1^2(C_2)$ have the fiber framings, i.e.,~the framings that are contained in $TF_a^{-1}(1/2)$.
Since $V(\eta)=1$, the restriction $V|_{H_1^i(C_i)}$ is transverse to the fiber $F_a^{-1}(1/2)$, and hence induces the fiber framing.

\begin{lemma}\label{lem:independence}
Let $C_{i,t}:=H_t^i(C_i)$. 
Then for any $t\in[0,1]$ and $p\in C_{i,t}$, the direction vector of $C_{i,t}$ at $p$ and $V(p)$ are linearly independent.
\end{lemma}

\begin{proof}
Let $v = \frac{dc}{ds}|_{s=0} \in T_pC_i$. 
For $t\in [0,1]$, the differential $(dH_t^i)_p(v)$ is calculated as follows: 
{\allowdisplaybreaks
\begin{align*}
(dH_t^i)_p(v) =& \left.\frac{d}{ds}\left(\phi_{-t \eta(c(s))}^V(c(s))\right)\right|_{s=0}\\
=& -t\left.\frac{d}{ds}\left(\eta(c(s))\right)\right|_{s=0}\left.\frac{d}{ds}\left(\phi_{s}^V(p)\right)\right|_{s=-t\eta(p)}+ \left.\frac{d}{ds}\left(\phi_{-t\eta(p)}^V(c(s))\right)\right|_{s=0}\\
=& \left(d\phi_{-t\eta(p)}^V\right)_p\left(v-tv(\eta)V(p)\right).
\end{align*}
}%
Since $(d\phi_{-t\eta(p)}^V)_p$ is an isomorphism and sends $V(p)$ to $V(\phi_{-t\eta(p)}^V(p))$, the linear independence of a direction vector of $C_{i,t}$ at $\phi_{-t\eta(p)}^V(p)$ and $V(\phi_{-t\eta(p)}^V(p))$ is equivalent to that of $v$ and $V(p)$.
Thus it suffices to show that $V$ is nowhere tangent to $C_i$. 

By direct calculation, we obtain
\[
V(r-\cos(2\pi s))=2\rho_r \sin(2\pi s), \quad V(y) = \frac{x\rho_z}{2(x^2+y^2)}. 
\]
Since $\sin (2\pi s)\neq 0$ at any point in $C_i$, $V$ is not tangent to $C_i$ on $\operatorname{supp}(\rho_r)$. 
On the other hand, since $y=0$ at any point in $C_i$, $V$ is not tangent to $C_i$ on $\operatorname{supp}(\rho_z)$. 
\end{proof}

We can deduce from this lemma that the framed curve $(C_i,V|_{C_i})$ is isotopic to $H_1^i(C_i)$ with the fiber framing.

\begin{lemma}\label{lem:framing comparison}
The framing $V|_{C_i}$ is homotopic to that given by $\partial_y$.
\end{lemma}

\begin{proof}
By the same calculation as that for \Cref{lem:independence}, we obtain:
\[
(dH_1^i)(\partial_y)
=
(d\phi_{-\eta}^V)\bigl(\partial_y - d\eta(\partial_y)V\bigr)
=
(d\phi_{-\eta}^V)(\partial_y - 2xV).
\]
Since $d\phi_{-\eta}^V(V)=V$ and $d\phi_{-\eta}^V$ induces a bijection from the framings of $C_i$ to those of $H_1^i(C_i)$, it suffices to show that $\Pa_y - 2x V$ is homotopic to $V$ as framings of $C_i$. 
It is indeed the case since $V$ and $\Pa_y$ are linearly independent in $\nu C_i$.
\end{proof}

By lifting the local description \eqref{eqn:local description hf} to that for $\widetilde{f}$ around $\widetilde{S}_k$, we can deduce that the two $2$-handles are attached with the framing $\partial_y$. 
In summary, we eventually obtain the following.

\begin{theorem}\label{thm:attcircle framing 2-handle tildeM-fibration}

Each of the circles $C_1$ and $C_2$ in \Cref{cor:handle decomposition tildeM-fibration} can be isotoped in $\widetilde{f}^{-1}(\partial \disk_{1/2-\epsilon})$ so that both circles are contained in a single fiber of $\widetilde{f}$. 
After the isotopy, their configuration is as illustrated in \Cref{fig:attcircle genus1surface}. 
Moreover, the framing of each $2$-handle agrees with the fiber framing.

\end{theorem}

\begin{remark}\label{rem:handle decomposition fibration over S2}

We can also obtain a handle decomposition of the total space of a genus-$g$ $M$-fibration $f':X'\to S^2$ by taking a handle decomposition of ${f'}^{-1}(D)$ in \Cref{cor:handle decomposition M-fibration}, where $D\subset S^2$ is a disk containing $\Critv(f)$, and then capping ${f'}^{-1}(D)$ by $D^2\times N_g$, which consists of one $2$-handle, $g$ $3$-handles and one $4$-handle. 

\end{remark}

\section{Examples of $M$-fibrations}

In this section, we present two examples of $M$-fibrations and study the topology of their total spaces. 
Using the handle decompositions obtained in the previous section, we describe Kirby diagrams of the orientation double coverings and derive applications to the existence problem of Lefschetz fibrations.

As explained in \Cref{sec:MCG}, $N_2^1$ is obtained from $\Sigma_0^3$ by gluing two copies of crosscaps. 
In particular, the surface described in the left-hand side of \Cref{fig:crosscap transposition} is $N_2^1$. 
Let $c,d\subset N_2^1$ be simple closed curves given in the figure.
Applying \Cref{cor:construction M-fibration S2 section} to the relation $u_{c,d}^2=t_{\delta,\theta}$, we obtain a genus-$2$ $M$-fibration $f_0:X_0\to S^2$ and a section $\sigma:S^2\to X_0$ of it such that $f_0$ has two $M$-singularities and the absolute value of the Euler number of the normal bundle of $\sigma(S^2)\subset X_0$ is $1$. 
Composing the orientation double covering $\mathcal{P}_0:\widetilde{X}_0\to X_0$, we also obtain a genus-$1$ $\widetilde{M}$-fibration $\widetilde{f}_0:=f_0\circ \mathcal{P}_0:\widetilde{X}_0\to S^2$ with two sections with self-intersection $1$ and $-1$, respectively. 

\begin{proposition}\label{prop:diffeo type tildeX0}

The manifold $\widetilde{X}_0$ is diffeomorphic to $L_2\sharp \mathbb{CP}^2\sharp \overline{\mathbb{CP}^2}$, where $L_2$ is the $4$-manifold given in \cite{Pao19774mfdtorusaction}.

\end{proposition}

\begin{proof}
By \Cref{cor:handle decomposition tildeM-fibration} together with \Cref{rem:handle decomposition fibration over S2}, the manifold $\widetilde{X}_0$ admits a handle decomposition consisting of one $0$-handle, two $1$-handles, six $2$-handles, two $3$-handles, and one $4$-handle. 
By \Cref{thm:attcircle framing 2-handle tildeM-fibration} and \Cref{fig:attcircle genus1surface}, among these six $2$-handles, the four arising from the $\widetilde{M}$-singularities are arranged as depicted in \Cref{fig:Kirby diagram tildeX0 1}, lying inside the rounded rectangular region formed by the four balls (the attaching regions of the $1$-handles) and the outer attaching circle. 
Note that, in this description, one must take care of the relative over/under information at the crossings of the two attaching circles corresponding to a single $\widetilde{M}$-singularity. 
As explained in the previous section, these circles are not originally contained in a single fiber; however, each can be isotoped into the same fiber. 
At each intersection, the over/under information is determined by the original curves before isotopy: the branch corresponding to the smaller value of $u$ (i.e.,~the imaginary part of the $z$-coordinate, defined in the previous section) lies above. 
\begin{figure}[htbp]
\subfigure[]{\includegraphics{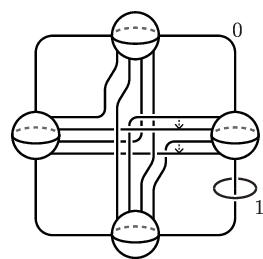}\label{fig:Kirby diagram tildeX0 1}}
\subfigure[]{\includegraphics{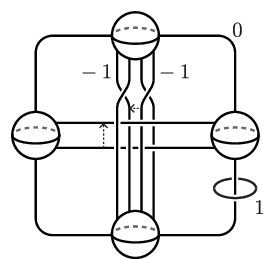}\label{fig:Kirby diagram tildeX0 2}}
\subfigure[]{\includegraphics{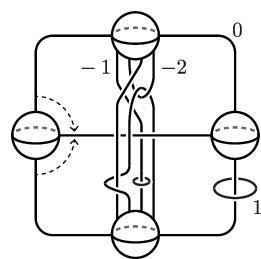}\label{fig:Kirby diagram tildeX0 3}}

\subfigure[]{\includegraphics{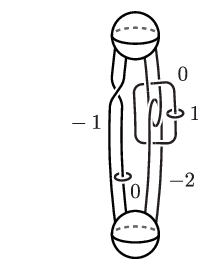}\label{fig:Kirby diagram tildeX0 4}}
\subfigure[]{\includegraphics{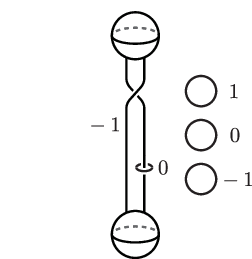}\label{fig:Kirby diagram tildeX0 5}}
\caption{The Kirby diagrams of the manifold $X_0$, where the number of $3$- and $4$-handles are omitted. }
\label{fig:Kirby diagram tildeX0}
\end{figure}

\Cref{fig:Kirby diagram tildeX0 2,fig:Kirby diagram tildeX0 3} are obtained from \Cref{fig:Kirby diagram tildeX0 1,fig:Kirby diagram tildeX0 2}, respectively, by sliding $2$-handles along the dashed arrows.
\Cref{fig:Kirby diagram tildeX0 4} is also obtained from \Cref{fig:Kirby diagram tildeX0 3} by performing two handleslides of a $2$-handle along the dashed arrows, followed by a cancellation of a pair of a $1$-handle and a $2$-handle, and then removing the linkings with the $(-1)$-framed attaching circle using the $0$-framed meridian. 
\Cref{fig:Kirby diagram tildeX0 5} is obtained from \Cref{fig:Kirby diagram tildeX0 4} by blowing down the $1$-framed unknot, and then the resulting $(-1)$-framed unknot, followed by an isotopy of a $2$-handle that moves it away from the $1$-handle. 
According to \cite[Figure~21]{Hayano2011genusSBLF}, the manifold described in \Cref{fig:Kirby diagram tildeX0 5} (with two $3$-handles and one $4$-handle) is $L_2\sharp \mathbb{CP}^2\sharp \overline{\mathbb{CP}^2}$. 
\end{proof}

The surface $N_3$ is obtained from $N_2^1$ by capping the boundary by a crosscap. 
In particular, one can regard $N_2^1$ (and thus simple closed curves $c,d\subset N_2^1$ in \Cref{fig:crosscap transposition}) as a subset of $N_3$. 
Since $u_{c,d}^2 = t_{\delta,\theta}$ in $\mathcal{M}(N_2^1)$ and a Dehn twist along a curve bounding a M\"obius band is trivial, $u_{c,d}^2$ is equal to $1$ in $\mathcal{M}(N_3)$. 
By \Cref{cor:construction M-fibration S2}, we obtain a genus-$3$ $M$-fibration $f_1:X_1\to S^2$ with two $M$-singularities. 
Again, composing the orientation double covering $\mathcal{P}_1:\widetilde{X}_1\to X_1$, we also obtain a genus-$2$ $\widetilde{M}$-fibration $\widetilde{f}_1:=f_1\circ \mathcal{P}_1:\widetilde{X}_1\to S^2$.

\begin{theorem}\label{thm:diffeo type tildeX1}

The manifold $\widetilde{X}_1$ admits no achiral Lefschetz fibration. 
Hence, the non-orientable $4$-manifold $X_1$ admits an $M$-fibration but no Lefschetz fibration.

\end{theorem}

\begin{proof}
By construction, $\widetilde{X}_1$ is obtained from $\widetilde{X}_0$ by removing neighborhoods of the two sections of $\widetilde{f}_0$ and gluing the total space of an annulus bundle over $S^2$. 
Since the total space being glued is diffeomorphic to $S^3 \times [0,1]$, it follows that $\widetilde{X}_1$ is obtained from $\widetilde{X}_0$ by blowing down the two sections and then performing surgery at the two points created by the blow-downs.
By \Cref{prop:diffeo type tildeX0}, we obtain $b_1(\widetilde{X}_0)=0$ and $b_2(\widetilde{X}_0)=2$. 
Hence, $b_1(\widetilde{X}_1)$ and $b_2(\widetilde{X}_1)$ are respectively equal to $1$ and $0$. 

If $\widetilde{X}_1$ admits an achiral Lefschetz fibration, we can deduce from \cite[Theorem~8.4.13]{GS19994-mfd} that $\widetilde{X}_1$ would admit a Lefschetz fibration. 
Since $b_2(\widetilde{X}_1)=0$, $\widetilde{X}_1$ admits no symplectic structure, and thus $\widetilde{X}_1$ is the total space of a torus bundle over $S^2$, contradicting the fact $\pi_1(\widetilde{X}_1)\cong \pi_1(\widetilde{X}_0)\ast \Z \cong (\Z/2\Z)\ast \Z$. 
The last statement immediately follows from the fact that the orientation double covering of a Lefschetz fibration is an achiral Lefschetz fibration (cf.~\cite{Yoshikawa2025oridoublecvrLF_preprint}). 
\end{proof}

\begin{remark}
By an argument analogous to the proof of \Cref{prop:diffeo type tildeX0}, namely, by drawing a Kirby diagram of $\widetilde{X}_1$ using \Cref{thm:attcircle framing 2-handle tildeM-fibration} and then transforming the diagram by handle slides and cancellations, one can also prove that $\widetilde{X}_1$ is diffeomorphic to $L_2'\sharp (S^1\times S^3)$, where $L_2'$ is also given in \cite{Pao19774mfdtorusaction}.
\end{remark}

\section{Problems and future directions}

The results of this paper lead to several natural questions concerning
$M$-fibrations and singular fibers on non-orientable $4$-manifolds.
We conclude this paper by listing some of them.

As discussed in Remark~3.9, the author expects that a suitable analogue of the
uniqueness theorem for Lefschetz fibrations should hold for $M$-fibrations.

\begin{problem}
Is the isomorphism class of an $M$-fibration determined by its monodromy
factorization (if the genus of an $M$-fibration is at least three)?
\end{problem}

\noindent
As explained in \Cref{rem:uniqueness M-fibration}, one difficulty is that it is not clear whether a crosscap transposition determines, in an appropriate sense, the isotopy classes of the curves used to define it.

Theorem~5.2 shows that there exist non-orientable $4$-manifolds admitting
$M$-fibrations but no Lefschetz fibrations. 
It would be interesting to understand to what extent $M$-fibrations and Lefschetz fibrations differ from the viewpoint of existence problems.

\begin{problem}
Which non-orientable $4$-manifolds admit $M$-fibrations?
In particular, does there exist a closed non-orientable $4$-manifold which
admits a Lefschetz fibration but does not admit an $M$-fibration?
\end{problem}

The mapping class group of a non-orientable surface is generated not only by Dehn twists but also by additional generators such as $Y$-homeomorphisms and crosscap transpositions. The present paper realizes crosscap transpositions as monodromies of singular fibers, while the corresponding problem for $Y$-homeomorphisms remains open.

\begin{problem}
Does there exist a singular fiber whose monodromy is a $Y$-homeomorphism?
\end{problem}

\appendix

\section{Singularity-theoretic properties of $M$-singularities}

In this appendix, we discuss several singularity-theoretic aspects of the local model
$F_m:M\times \mathbb{C}\to \mathbb{C}$ of an $M$-singularity introduced in \Cref{sec:monodromy M-fibration}.
More specifically, we compute the $\mathcal{A}_e$-codimension of the map-germ $F_m$ at an $M$-singularity, and construct an explicit perturbation of $F_m$ into a stable map.
We use several standard notions and notations from singularity theory, including $\mathcal{E}_n^p$, $tf$, $\omega f$, and $\mathcal{A}_e$-codimension.
For precise definitions, we refer the reader to, for example, \cite{Wall1981findet}. 

\subsection*{$\mathcal{A}_e$-codimension}

We first study the $\mathcal{A}_e$-codimension of the map-germ associated with $F_m$.
Since the quotient map $\pi_m:\R^2 \to M$ is a local diffeomorphism, it is enough to consider the map-germ $f:(\R^4,0)\to(\R^2,0)$ defined by $f(r,s,x,y)=(r\cos s+x^2-y^2,r\sin s+2xy)$, where $z=x+\sqrt{-1}y$.

\begin{proposition}
The map-germ $f$ has infinite $\mathcal{A}_e$-codimension.
\end{proposition}

\begin{proof}
The $\mathcal{E}_4$-module $\mathcal{E}_4^2$ admits the following direct sum decomposition.  
\[
\mathcal{E}_4^2 \cong \left<(\cos s,\sin s)\right>_{\mathcal{E}_4}\oplus \left<(-\sin s,\cos s)\right>_{\mathcal{E}_4}\cong \mathcal{E}_4\oplus \mathcal{E}_4.
\]
Let $\rho:\mathcal{E}_4^2\to \mathcal{E}_4$ be the projection to the second component. 
By direct calculation, one can check that the following formulas hold: 
\[
\rho(tf(e_1)) = r,\quad \rho(tf(e_2)) = 0,\quad\rho(tf(e_3)) = 2(y\cos s -x\sin s), \quad\rho(tf(e_4)) = 2(x\cos s+y\sin s). 
\]
We thus obtain the isomorphism $\mathcal{E}_4^2/tf(\mathcal{E}_4^4) \cong \mathcal{E}_4/\left<r,x,y\right>_{\mathcal{E}_4} \cong \mathcal{E}_1$. 
Since $\rho(f)$ is contained in $\left<r,x,y\right>_{\mathcal{E}_4}$, we obtain the following isomorphism as $\R$-vector spaces:
\[
\mathcal{E}_4^2/tf(\mathcal{E}_4^4)\supset T\mathcal{A}_e f/tf(\mathcal{E}_4^4) \cong \left<\cos s,\sin s\right>_{\R}\subset \mathcal{E}_1.  
\]
By the third isomorphism theorem, we obtain:
\[
\dim_{\R}\left(\frac{\mathcal{E}_4^2}{T\mathcal{A}_ef}\right) = \dim_{\R}\left(\frac{\mathcal{E}_4^2/tf(\mathcal{E}_4^4)}{T \mathcal{A}_ef/tf(\mathcal{E}_4^4)}\right) = \dim_{\R}\left(\frac{\mathcal{E}_1}{\left<\cos s,\sin s\right>_{\R}}\right)=\infty. \qedhere
\]
\end{proof}


\subsection*{Perturbation into a stable map}

We next give a stable perturbation of $F_m$. 
For $\epsilon >0$, we define $\mathcal{F}_{\varepsilon}:M\times \C\to \C$ by $\mathcal{F}_{\varepsilon}([r,s]_m,z)=re^{\pi i s}+z^2+\varepsilon e^{4\pi i s}$. 
Taking the coordinates $(r,\theta=\pi s, x,y)$ ($\theta\in [0,\pi]$) and $\C\cong \R^2$, we have 
\[
\mathcal{F}_{\varepsilon}(r,\theta,x,y)=(r\cos\theta+x^2-y^2+\varepsilon\cos4\theta,r\sin\theta+2xy+\varepsilon\sin4\theta).
\]

\begin{proposition}

For a sufficiently small \(\varepsilon>0\), the critical point set of \(\mathcal{F}_{\varepsilon}\) is a circle consisting only of folds and three cusps.
Moreover, its image has no self-intersections.
Consequently, \(\mathcal{F}_{\varepsilon}\) is infinitesimally stable.

\end{proposition}

\noindent
In particular, this proposition shows that the critical value set
of the stable perturbation above is topologically equivalent to
the one arising from a wrinkling, which is a stable perturbation of a Lefschetz
singularity introduced in \cite{Lekili2009wrinkledfibration}.

\begin{proof}
We will use Saji's criteria \cite{Saji2016criteriaMorin} for folds and cusps.
In what follows, we adopt the notation used in \cite{Saji2016criteriaMorin}.
By direct calculation, we obtain:
{\allowdisplaybreaks
\begin{align*}
&d\mathcal{F}_{\varepsilon}(\partial_r)
=
(\cos\theta,\sin\theta),\quad
d\mathcal{F}_{\varepsilon}(\partial_x)
=
(2x,2y),\quad
d\mathcal{F}_{\varepsilon}(\partial_y)
=
(-2y,2x),\\
&d\mathcal{F}_{\varepsilon}(\partial_\theta)
=
(-r\sin\theta-4\varepsilon\sin4\theta,
\,
r\cos\theta+4\varepsilon\cos4\theta).
\end{align*}
}%
Since $d\mathcal{F}_\varepsilon(\Pa_r)\neq 0$, the rank of $d\mathcal{F}_{\varepsilon}$ is at least $1$ everywhere.
Furthermore, calculating the determinants, one can check that $\Crit(\mathcal{F}_\varepsilon)$ is equal to $\{(r,\theta,0,0)~|~ r= -4\varepsilon\cos 3\theta\}$. 

We put
$\xi_1=\partial_r,\eta_1=\partial_x,\eta_2=\partial_y,\eta_3=\partial_\theta+4\varepsilon\sin3\theta\,\partial_r$.
It is easy to check that these vector fields are adapted with respect to the map-germ $\mathcal{F}_\varepsilon$ at any critical point. 
Let $\lambda_i=\det\bigl(d\mathcal{F}_{\varepsilon}(\xi_1),d\mathcal{F}_{\varepsilon}(\eta_i)\bigr)$ for $i=1,2,3$, which are calculated as follows:
{\allowdisplaybreaks
\begin{align*}
\lambda_1
=&
\det
\begin{pmatrix}
\cos\theta & 2x\\
\sin\theta & 2y
\end{pmatrix}
=
2(y\cos\theta-x\sin\theta),\\
\lambda_2
=&
\det
\begin{pmatrix}
\cos\theta & -2y\\
\sin\theta & 2x
\end{pmatrix}
=
2(x\cos\theta+y\sin\theta),\\
\lambda_3
=&
\det
\begin{pmatrix}
\cos\theta &
-r\sin\theta-4\varepsilon\sin4\theta
\\
\sin\theta &
r\cos\theta+4\varepsilon\cos4\theta
\end{pmatrix}=r+4\varepsilon\cos3\theta.
\end{align*}}
The Jacobi matrix of $\Lambda=(\lambda_1,\lambda_2,\lambda_3):\R^4\to\R^3$ at a critical point $(r=-4\varepsilon \cos3\theta, \theta,0,0)$ is 
\[
\begin{pmatrix}
0&0&-2\sin\theta&2\cos\theta\\
0&0&2\cos\theta&2\sin\theta\\
1&-12\varepsilon\sin3\theta&0&0
\end{pmatrix},
\]
whose rank is \(3\). 
Hence every critical point of $\mathcal{F}_\varepsilon$ is non-degenerate in the sense
of \cite[Definition 2.3]{Saji2016criteriaMorin}.

By the direct calculation, one can show that the following equality holds at a critical point $(-4\varepsilon \cos3\theta,\theta,0,0)$:
\[
\mathcal H:=(\eta_i\lambda_j)_{1\le i,j\le3}=\begin{pmatrix}
-2\sin\theta & 2\cos\theta & 0\\
2\cos\theta & 2\sin\theta & 0\\
0&0&-8\varepsilon\sin3\theta
\end{pmatrix}.
\]
Hence we obtain $H:=\det\mathcal H
=
32\varepsilon\sin3\theta$ and a critical point $(-4\varepsilon \cos3\theta,\theta,0,0)$ is $2$-singular in the sense of \cite{Saji2016criteriaMorin} if and only if $\sin 3\theta =0$. 
By \cite[Theorem 3.1]{Saji2016criteriaMorin}, $(-4\varepsilon \cos3\theta,\theta,0,0)$ is a fold if and only if $\sin 3\theta \neq 0$.

Suppose now that $\sin3\theta=0$. 
Then \(H=0\), so the point $p=(-4\varepsilon \cos3\theta,\theta,0,0)$ is \(2\)-singular. 
Since $S(\mathcal{F}_\varepsilon)=\Crit(\mathcal{F}_{\varepsilon})
=
\{(r,\theta,0,0)\mid r=-4\varepsilon\cos3\theta\}$, we may regard \(H|_{S(\mathcal{F}_{\varepsilon})}\) as a function of \(\theta\) and we obtain $d(H|_{S(\mathcal{F}_{\varepsilon})})_p=
96\varepsilon\cos3\theta$. 
Since $\sin3\theta=0$, $\cos3\theta=\pm1$ and thus every \(2\)-singular point is \(2\)-non-degenerate.
Moreover, since $S_2(\mathcal{F}_{\varepsilon})$ is a finite set, any
\(2\)-non-degenerate point cannot be \(3\)-singular. 
By \cite[Theorem 3.1]{Saji2016criteriaMorin}, every \(2\)-singular point is a cusp.

As explained above, $\Crit(\mathcal{F}_\varepsilon)$ is parameterized by $\theta$, and so is $\Critv(\mathcal{F}_\varepsilon)$. 
the image of the critical point $(-4\varepsilon \cos3\theta,\theta,0,0)$ by $\mathcal{F}_\varepsilon$ is
\[
\gamma(\theta)
=
\bigl(
- \varepsilon\cos4\theta-2\varepsilon\cos2\theta,
\,
-\varepsilon\sin4\theta+2\varepsilon\sin2\theta
\bigr)=-\varepsilon e^{4i\theta}-2\varepsilon e^{-2i\theta}.
\]
Suppose $\gamma(\theta_1)=\gamma(\theta_2)$ for some $\theta_1,\theta_2\in [0,\pi]$ with $\theta_1\neq \theta_2$. 
Putting $w_i=e^{2i\theta_i}$, we obtain 
\[
w_1^2+2w_1^{-1}=w_2^2+2w_2^{-1}\Leftrightarrow(w_1-w_2)
\left(
w_1+w_2-\frac{2}{w_1w_2}
\right)
=
0.
\]
Since \(|w_1|=|w_2|=1\), the second factor vanishes only if
\(w_1=w_2\). 
Hence the critical value set has no self-intersections.
\end{proof}

\subsection*{Acknowledgment}

The author would like to thank Yusuke Kuno for asking whether a Y-homeomorphism can be realized as the monodromy of a singular fiber, which motivated the present work. 
The author is also grateful to Susumu Hirose for several valuable comments. 
The author was supported by Japan Society for the Promotion of Science KAKENHI Grant Number JP23K03123.


\end{document}